\documentclass[reqno, 11pt, a4paper]{amsart}

\usepackage{latexsym,ifthen,xspace}
\usepackage{amsmath,amssymb,amsthm}
\usepackage{enumerate, calc}
\usepackage[colorlinks=true, pdfstartview=FitV, linkcolor=blue,
  citecolor=blue, urlcolor=blue, pagebackref=false]{hyperref}
\usepackage[text={33pc,605pt},centering]{geometry}    %11pt
\usepackage{graphicx,color}
\usepackage[sort]{cite}
\newtheorem{theorem}{Theorem}[section]
\newtheorem{lemma}[theorem]{Lemma}
\newtheorem{proposition}[theorem]{Proposition}
\newtheorem{assumption}[theorem]{Assumption}
\newtheorem{corollary}[theorem]{Corollary}

\theoremstyle{definition}
\newtheorem{definition}[theorem]{Definition}
\newtheorem{example}[theorem]{Example}

\theoremstyle{remark}
\newtheorem{remark}[theorem]{Remark}

\numberwithin{equation}{section}

\DeclareMathAlphabet{\mathsl}{OT1}{cmss}{m}{sl}
\SetMathAlphabet{\mathsl}{bold}{OT1}{cmss}{bx}{sl}

\newcommand{\si}{\ensuremath{\sigma}}

\newcommand{\om}{\ensuremath{\omega}}

\newcommand{\Om}{\ensuremath{\Omega}}
\newcommand{\cF}{\ensuremath{\mathcal F}}

\newcommand{\cL}{\ensuremath{\mathcal L}}

\newcommand{\cR}{\ensuremath{\mathcal R}}

\newcommand{\bbN}{\ensuremath{\mathbb N}}

\newcommand{\bbR}{\ensuremath{\mathbb R}}

\newcommand{\bbZ}{\ensuremath{\mathbb Z}} 
\newcommand{\ldef}{\ensuremath{\mathrel{\mathop:}=}}

\definecolor {orange} {rgb} {0.569, 0.259, 0.0}

\DeclareMathOperator{\mean}{\mathbb{E}}

\DeclareMathOperator{\prob}{\mathbb{P}}

\DeclareMathOperator{\cov}{\mathbb{C}ov}

\def\indicator{{\mathchoice {1\mskip-4mu\mathrm l}%
{1\mskip-4mu\mathrm l}{1\mskip-4.5mu\mathrm l}%
{1\mskip-5mu\mathrm l}}}

\newcommand\numberthis{\addtocounter{equation}{1}\tag{\theequation}}

\usepackage{amsfonts,amsmath,amssymb,amsthm}
\usepackage[utf8]{inputenc}
\usepackage[english]{babel}
\usepackage{bbm}
\usepackage{enumitem}
\usepackage{color}
\usepackage{accents}
\usepackage{comment}

\DeclareMathOperator{\Z}{\mathbb{Z}}
\DeclareMathOperator{\R}{\mathbb{R}}
\DeclareMathOperator{\N}{\mathbb{N}}
\DeclareMathOperator{\pr}{\mathbb{P}}

\newcommand{\ubar}[1]{\underaccent{\bar}{#1}}

\newcommand{\abs}[1]{\left\lvert#1\right\rvert}
\newcommand{\norm}[1]{\left\lVert#1\right\rVert}

\newcommand{\E}[1]{\mathbb{E}\left[#1\right]}

\newcommand{\RR}[3]{\mathcal{R}_{#1}^{#2,#3}}

\begin{document}

\title[Lower Gaussian heat kernel bounds for the RCM]{Lower Gaussian heat kernel bounds for the Random Conductance Model in a degenerate ergodic environment}

%    Remove any unused author tags.

\author{Sebastian Andres}
\address{The University of Manchester}
\curraddr{Department of Mathematics,
Oxford Road, Manchester M13 9PL}
\email{sebastian.andres@manchester.ac.uk}
\thanks{}

%    author two information
\author{Noah Halberstam}
\address{University of Cambridge}
\curraddr{Centre for Mathematical Sciences, Wilberforce Road, Cambridge CB3 0WB}
\email{nh448@cam.ac.uk}
\thanks{}



\date{\today}

\dedicatory{}

\begin{abstract}
We study the random conductance model on  $\Z^d$ with ergodic, unbounded conductances. We prove a Gaussian lower bound on the heat kernel given a polynomial moment condition  and some additional assumptions on the correlations of the conductances. The proof is based on the well-established chaining technique. We also obtain bounds on the Green's function.
\end{abstract}

\maketitle

\section{Introduction}

\subsection{The Model}

We let $G=(\Z^d,E_d),$ where $E_d=\{\{x,y\}\in\Z^d\times\Z^d: |x-y|=1\},$ be the $d$-dimensional lattice for a fixed dimension $d\geq 2$. We write $x\sim y$ if $(x,y)\in E_d$. We consider the space of positive weightings on the edges of the graph, $\Omega=(0,\infty)^{E_d},$ and for $\omega\in\Omega,$ we access the weight at a particular edge $e\in E_d$ by $\omega(e)$, which we will also refer to as the \emph{conductance} on an edge $e$.   For $x,y\in\Z^d$ and $\omega\in\Omega$ we set $\omega(x,y)=\omega(y,x)=\omega(\{x,y\})$ if $\{x,y\}\in E_d,$ else $\omega(x,y)=0.$
For any fixed $\omega,$ we define measures $\mu^\omega$ and $\nu^\omega$ on $\Z^d$ by 
\[
\mu^\omega(x)\ldef \sum_{y\sim x}\omega(x,y)\qquad \text{and} \qquad \nu^\omega(x) \ldef \sum_{y\sim x}\frac{1}{\omega(x,y)}.
\]
For any $z \in \bbZ^d$ we denote by $\tau_z\!: \Om \to \Om$  the space shift by $z$ defined by
\begin{align*}
  (\tau_{z}\, \om)(\{x,y\})
  \;\ldef\;
  \om(\{x+z,y+z\}),
  \qquad \forall\; \{x,y\} \in E_d.
\end{align*}
We equip $\Om$ with a $\si$-algebra $\cF$.  Further, we will denote by $\prob$ a probability measure on $(\Om, \cF)$, and we write $\mean$ for the expectation with respect to $\prob$. Throughout the paper we will assume that the conductances are stationary and ergodic. 

\begin{assumption}[Stationarity and ergodicity] \label{asm:ergodic}
  $\prob$ is stationary and ergodic with respect to translations of $\bbZ^d$, i.e. $\prob \circ\, \tau_x^{-1} \!= \prob\,$ for all $x \in \bbZ^d$ and $\prob[A] \in \{0,1\}\,$ for any $A \in \cF$ such that $\tau_x(A) = A\,$ for all $x \in \bbZ^d$.
\end{assumption}

We now introduce the \emph{random conductance model (RCM)}. For a given $\om \in \Omega$, we  consider the continuous time Markov chain  $X=\{X_t: t\geq 0\}$ on $\bbZ^d$ with generator
\[
(\mathcal{L}^\omega f)(x)=\frac{1}{\mu^\omega(x)}\sum_{y\sim x}\omega(x,y) \, \big(f(y)-f(x) \big).
\]
This stochastic process, also known as the \emph{constant speed random walk} (CSRW), waits at $x$ for an exponential time with mean $1$, and then chooses the next position $y\sim x$ with probability $\omega(x,y)/\mu^\omega(x)$.  We also recall that the Markov chain $X$ is reversible with respect to $\mu^\omega$. We denote by $\mathrm{P}_x^\omega$ the law of the walk starting at the vertex $x\in\Z^d,$ and by $\mathrm{E}_x^\omega$ the expectation with respect to this law.  For $x,y \in \Z^d$ and $t>0,$ we let $p^\omega(t,x,y)$ be the transition density (or the heat kernel associated with $\cL^\omega$) with respect to the measure $\mu^\omega$, i.e.
\[
p^\omega(t,x,y)=\frac{\mathrm{P}_x^\omega[X_t=y]}{\mu^\omega(y)}.
\]

\subsection{Main Results}
The random conductance model has been the subject of extensive research for more than a decade, 
see \cite{Bi11, Ku14} for surveys of the model and references therein.
 More recent results include the derivation of quenched functional central limit theorems \cite{ADS15, BS19, DNS18, BCKW20} and local limit theorems \cite{BH09,harnack, ACS20, AT19, BS20} for the RCM with unbounded ergodic conductances under moment conditions.
 In this paper we will focus on heat kernel estimates, see  e.g.\ \cite{De99, Ba04, BBHK08, BH09, BD10, Fo11, BB12, ADS16a,ADS19} for previous results. 
In particular, we will obtain Gaussian type lower bounds on the heat kernel in the case of  ergodic unbounded conductances.

It is known that Gaussian bounds do not hold in general: for example, under i.i.d.\ conductances with fat tails at zero, the heat kernel decay may be sub-diffusive due to a trapping phenomenon -- see \cite{BBHK08, BB12}. Moreover, in \cite[Theorem~5.4]{harnack}, it is proved that in the general ergodic setting, moment bounds on the conductances and their reciprocals are a necessary condition for upper and lower near-diagonal Gaussian bounds to hold. In \cite{ADS16a}, this necessary condition is shown to be sufficient 
for full upper Gaussian heat kernel bounds.

Gaussian lower bounds have been shown on i.i.d.\ percolation clusters in \cite{Ba04}, and for variable speed random walks under i.i.d.\ conductances bounded away from zero in \cite{BD10}.
However, in the general ergodic setting, as of yet, Gaussian lower bounds have only been proved under the stronger condition of uniformly elliptic conductances \cite{De99}, i.e.\ $c^{-1} \leq \om (e) \leq c$, $e\in E_d$, for some $c \geq 1$. In this paper we relax the uniform ellipticity assumption, substituting it for the combination of a polynomial moment condition together with an assumption concerning the correlations of the conductances, see Assumption~\ref{asm:mixing}.  It is unknown whether moment conditions by themselves should be sufficient for the lower bound to hold. The main available technique for proving lower bounds,  the \textit{chaining method} (see \cite{FS86}), fails at present in this generality (see Section~\ref{sec:intro_method} below for a more in-depth discussion), while our assumptions are sufficient to ensure the functionality of this method. However, it seems that other techniques would be required in order to weaken these assumptions. One possible approach would be to use techniques from quantitative stochastic homogenization that lead to much stronger quantitative homogenization results for heat kernels and Green functions, see \cite[Chapters 8--9]{AKM19} for details. This technique has been adapted to Bernoulli bond percolation clusters in \cite{dario2020quantitative}, and it is expected that it also applies to other degenerate models.

\medskip
We will begin by recalling the already established Gaussian upper bound in \cite{ADS16a}, for which we will need some more notation. For $A\subset\Z^d$ non-empty and finite, and $p\in[1,\infty)$, we introduce space-averaged $l^p$ norms on functions $\phi:A\rightarrow \R$ by 
\[\norm{\phi}_{p,A}\ldef \bigg(\frac{1}{\abs{A}}\sum_{x\in A} \big| \phi(x)\big|^p\bigg)^{\! \frac{1}{p}}\qquad\text{and}\qquad \norm{\phi}_{\infty,A}\ldef \max_{x\in A} \abs{\phi(x)},
\]
where $\abs{A}$ denotes the cardinality of the set $A$.  For $x \in \bbZ^d$ we denote by $B(x,r) \ldef \{y \in \bbZ^d : |x-y| < r \}$ balls in $\bbZ^d$ centered at $x$ with respect to the graph distance, where  $|x|\ldef \sum_{i=1}^d|x_i|$ for $x=(x_1,\ldots, x_d)\in \bbZ^d$.
Suppose now that $\om(e)\in L^p(\prob)$ and $\om(e)^{-1} \in L^q(\prob)$ for any $p,q\geq 1$. Then, under Assumption~\ref{asm:ergodic}, the spatial ergodic theorem gives that, $\prob$-a.s., for any $x\in \bbZ^d$,
\[
\bar{\mu}_p\ldef  \mean[ \mu^\om(0)^p]=\lim_{n\rightarrow\infty}\norm{\mu^\omega}_{p,B(x,n)}^p, \qquad  \bar{\nu}_q \ldef  \mean[ \nu^\om(0)^q] =\lim_{n\rightarrow\infty}\norm{\nu^\omega}_{q,B(x,n)}^q.
\] 
In particular, for $\prob$-a.e.\ $\om$ and each $x\in\Z^d,$ there exists $N_1(x)=N_1( \omega,x,p,q)\in\N$ such that
\begin{align} \label{eq:defN}
\sup_{n\geq N_1(x)} \norm{\mu^\omega}_{p,B(x,n)}^p\leq2\bar{\mu}_p,
\qquad \sup_{n\geq N_1(x)} \norm{\nu^\omega}_{q,B(x,n)}^q\leq2\bar{\nu}_q.
\end{align}
We will choose $N_1(x)$ to be the minimal such random variable. The Gaussian upper heat kernel bound is as follows:

\begin{theorem} \label{thm:uhk}
  Suppose that Assumption~\ref{asm:ergodic} holds and suppose there exist $p,q\in (1,\infty]$ with $1/p+1/q<2/d$ such that  $\omega(e)\in L^p(\pr)$ and $\omega(e)^{-1}\in L^q(\pr)$ for any $e\in E_d$.  Then,  there exist  constants $c_i = c_i(d, p, q, \bar{\mu}_p, \bar{\nu}_q)$   such that,  for $\prob$-a.e.\ $\om$, for any given $t$ and $x$  with $\sqrt{t} \geq N_1(x)$ and all $y \in \bbZ^d$ the following hold.
  
  \begin{enumerate}
  \item [(i)] If $|x-y|\leq c_1 t$ then
    \begin{align*}
      p^\om(t,x, y)
      \;\leq\;
      c_2\, t^{-d/2}\,  \exp\!\big(\!-c_3\, |x-y|^2/t\big).
    \end{align*}
  \item [(ii)] If $|x-y|\geq c_1 t$ then
    \begin{align*}
      p^\om(t,x, y)
      \;\leq\;
      c_2 \, t^{-d/2} \exp\!\big(\!-c_4\, |x-y| (1 \vee \log(|x-y|/t))\big).
    \end{align*}
  \end{enumerate}
   
\end{theorem}
\begin{proof}
See \cite[Theorem~1.6]{ADS16a} and a more general version with a streamlined proof in \cite[Theorem~3.2]{ADS19}.
\end{proof}

We now state the additional assumptions we require, followed by our main results. We will then discuss why these additional assumptions are needed and how they interact with the strategy of the proof.

 For $\om,\om' \in \Omega$ we write $\om \leq \om'$ if $\om(e) \leq \om'(e)$ for all $e\in E_d$. We say that a function $f:\Omega \rightarrow \bbR$ is non-decreasing if $f(\om) \leq f(\om')$ whenever $\om \leq \om'$.
 	
\begin{assumption}\label{asm:mixing} At least one of the following four conditions holds.
\begin{enumerate}  \setlength{\itemsep}{6pt}
	\item[(A1)] 
	(i) \emph{FKG inequality.} For any finite set of edges $A\subset E_d,$ and any non-decreasing functions $f,g:\Omega\rightarrow\R$ depending only on $\{\omega(e):e\in A\}$, we have
		\begin{equation} \label{eq:FKG}
		\cov(f,g)\geq 0,
		\end{equation} whenever the covariance exists.

	\medskip
	 
		 \noindent (ii) \emph{Polynomial mixing.} There exist constants $\gamma>d^2$ and $C_{\mathrm{mix}}\in (0,\infty)$ such that for any non-decreasing function $f \in L^2(\prob)$ depending only 
		on $\{\om(0,y), |y|=1\}$, and any $x\in\Z^d\setminus\{0\}$, 
		\begin{align*}
		\cov\big(f, f\circ \tau_x \big) \leq C_{\mathrm{mix}} \norm{f}_{L^2(\pr)}^2  \abs{x}^{-\gamma}.
		\end{align*}		
		
	\item[(A2)] \emph{Spectral gap.}
	There exists $C_{\mathrm{sg}}\in(0,\infty)$ such that
	\begin{align}\label{SG}
	\mean\big[(f-\mean[f])^2\big] \leq C_{\mathrm{sg}} \sum_{e\in E_d} \mean \Big[ \big( \partial_{e} f \big)^2 \Big],
	\end{align}
	for any  $f \in L^2(\prob)$. Here, the `vertical derivative' $\partial_e f$ is defined as
	\begin{equation*}
	\partial_{e} f(\omega):=\limsup\limits_{h\to 0}\frac{f(\omega+h\delta_e)-f(\omega)}{h},
	\end{equation*}
	where $\delta_e:E_d\to\{0,1\}$ stands for the Dirac function satisfying $\delta_e(e)=1$ and $\delta_e(e')=0$ if $e'\not= e$. 
	
	\item[(A3)] \emph{Finite range dependence.} There exists a positive constant $\mathfrak{R} \in \bbN$, such that for any $x\in\Z^d$, the collection of random variables $\big(\omega(\{x,x+e\}):\abs{e}=1\big)$ is independent of  $\big(\omega(\{z,z+e\}):\abs{z-x}\geq\mathfrak{R},\abs{e}=1\big)$.
	
	\item[(A4)] \emph{Negative association.} For any finite set of edges $A\subset E_d,$ and any non-decreasing  functions $f,g:\Omega\rightarrow\R$ depending only on $\{\omega(e):e\in A\}$, we have 
	\[
	\cov(f,g)\leq 0,
	\] whenever the covariance exists.
\end{enumerate}

\end{assumption}

The FKG inequality was first (formally) investigated by Fortuin, Kastelyn and Ginibre in  \cite{fortuin1971} in connection with  correlation properties of Ising  spin systems. The inequality is in fact a natural property of a very wide range of statistical mechanics models, including the random cluster model (with $q\geq1$) \cite{RandomCluster}, Yukawa quantum field theory models \cite{newman1980}, Gaussian free fields \cite[Proposition 5.22]{GFF} and interlacement percolation \cite{interlacements}. We note that by \cite[Theorem 3.3]{applicationsassoc}, it is sufficient to check that \eqref{eq:FKG} holds for bounded continuous non-decreasing functions. On the other hand, the opposite assumption of negative association (A4) also holds for some prominent models, including the uniform spanning tree, the random cluster model (with $q\leq 1$), and simple exclusion models, we refer to \cite{Pemantle} also for more motivation and background for this condition. We note that in the case of Gaussian fields, pairwise positive and negative correlation are enough to imply the FKG inequality \cite{NormalAssoc} and negative associativity \cite{joag-dev1983}, respectively.

The spectral gap  condition in (A2) and the finite range dependence in (A3) also appear as decorrelation assumptions in the context of quantitative stochastic homogenization, see for instance \cite{GNO15} for (A2), and \cite{AKM19} for (A3). In a sense, the spectral gap condition in (A2), introduced in \cite{GNO15}, can be interpreted as a quantified  version of ergodicity, as it implies an optimal variance decay for the semigroup associated with the ``process of the environment as seen from the particle'' induced by the simple random walk on $\bbZ^d$, cf.\ \cite[Proposition~1 and Remark~5]{GNOlong}.

\medskip
Throughout the paper we write $c$ to denote a positive constant which may change on each appearance,  while constants denoted $c_i$ will be the same through the paper. 
The constants will depend only on $d,p,q$, the moments of $\mu^\om(0)$ and $\nu^\om(0)$, and the parameters $\gamma,C_{\mathrm{mix}},C_{\mathrm{sg}},\mathfrak{R}$ in Assumption~\ref{asm:mixing} as appropriate, unless the dependencies are specified in the particular context.

\begin{theorem} 	\label{thm:main}
	Let $d\geq 2$ and suppose that Assumptions \ref{asm:ergodic} and \ref{asm:mixing} hold. Then there exist constants  $c_5, c_6, c_7 \in (0,\infty)$ and $p_0,q_0 \in [1,\infty)$ such that if $\omega(e)\in L^{p_0}(\pr)$ and $\omega(e)^{-1}\in L^{q_0}(\pr)$ the following holds.
	For $\prob$-a.e.\ $\om$ and any $x\in\Z^d$, there exists a random constant $N(x)=N(\om,x)$ satisfying
	\begin{equation} \label{eq:mainTailBound}
	\pr(N(x)> r) \leq c_5 \, r^{-\alpha}, \qquad \forall r>0,
	\end{equation}
	for some $\alpha >d(d-1)-2$, such that for all $y \in \Z^d$ and $t\geq N(x) (1\vee |x-y|)$,
	\begin{equation} \label{eq:mainBound}
	p^{\omega}(t, x, y) \geq c_6 \,  t^{-d/2} \exp\!\big(-c_7 |x-y|^{2} / t \big).
	\end{equation}
\end{theorem}

\begin{remark} \label{rem:min}
(i) Minimal choices for $p_0$ and $q_0$ are
 $p_0>p\kappa\chi$ and $q_0>q\kappa\chi$ with $\chi:=d^2(1+\frac{d^2-2}{\gamma-d^2})$ under (A1), and $p_0=2p\kappa d$ and $q_0=2q\kappa d$ under (A2), (A3) or (A4),
  for any $p,q>1$ satisfying $1/p+1/q<2/d$, and with $\kappa=\kappa(p,q,d)$ as in Proposition~\ref{prop:harnack} below. 
 More precisely,  the quantity $\kappa$ originally appears in the random constant of the parabolic Harnack inequality in \cite{harnack}, which serves as one main ingredient in the proof of Theorem~\ref{thm:main}. 
 
 (ii) Given the two-sided heat kernel bounds provided by Theorems~\ref{thm:uhk} and \ref{thm:main}, the law of iterated logarithm (LIL) for the sample paths of the random walk can be established, see \cite{DC08, KN16}. However, it is expected that the LIL can be derived more easily under much weaker assumptions by exploiting the decomposition of the random walk into a martingale part and a corrector function, used in many proofs of a quenched functional central limit theorem, together with  the sublinearity of the corrector  (see e.g.\ \cite{ABDH13, ADS15, BS20, BCKW20}) and an LIL for the martingale part.
 \end{remark}

In $d\geq 3$, we can use Theorems \ref{thm:main} and \ref{thm:uhk} to derive the following bound on the Green kernel, $g^\omega(x,y),$ defined by
\[
g^\omega(x,y)\ldef \int_{0}^{\infty} p_t^\omega(x,y)\	 \textrm{d}t, \qquad x,y\in \bbZ^d.
\]
We refer to \cite[Theorem 1.2]{ABDH13} for precise estimates
and asymptotics in the case of general non-negative i.i.d.\ conductances, to
\cite[Theorem 1.14]{harnack} for a local limit theorem for $g^\om$ in the case of ergodic conductances satisfying a moment condition, and to \cite{ADS20} for recent results on the Green kernel in dimension $d=2$.

\newpage

\begin{theorem} \label{thm:green}
Let $d\geq 3$ and suppose that Assumption~\ref{asm:ergodic} holds. 
\begin{enumerate}
\item[(i)] Suppose there exist $p,q\in (1,\infty]$ with $1/p+1/q<2/d$ such that  $\omega(e)\in L^{p}(\pr)$ and $\omega(e)^{-1}\in L^{q}(\pr)$ for any $e\in E_d$. 
For $\prob$-a.e.\ $\om$, there exist $c_8\in (0,\infty)$ and a random constant $N_2(x)=N_2(\om,x)$ such that for all $x,y\in\Z^d$ with $\abs{x-y}\geq N_2(x)$, 
\begin{equation}  \label{eq:GU}
g^\omega(x,y)\leq c_8\abs{x-y}^{2-d}.
\end{equation}
\end{enumerate}
Additionally, suppose that Assumption~\ref{asm:mixing} is satisfied. Then there exist $c_9, c_{10}, c_{11} \in(0,\infty)$ and $p_0,q_0\in[1,\infty)$ such that if $\omega(e)\in L^{p_0}(\prob)$ and $\omega(e)^{-1}\in L^{q_0}(\prob)$, then the following hold.
\begin{enumerate}
\item[(ii)] For all $x,y\in\Z^d$ with $\abs{x-y}>N(x)$, 
\begin{equation} \label{eq:GL}
g^\omega(x,y)\geq c_9 \, |x-y|^{2-d}.
\end{equation}

\item[(iii)] For any $x,y\in\Z^d$ with $x\not=y$,
\begin{align} \label{eq:AGL}
c_{10} \, \abs{x-y}^{2-d} \leq \mean\big[ g^\omega(x,y) \big] \leq c_{11} \, \abs{x-y}^{2-d}.
\end{align}

\end{enumerate}
\end{theorem} 

\medskip

\begin{example}[RCMs defined by Ginzburg-Landau $\nabla\phi$ interface models]
One class of conductances satisfying the assumptions of Theorem~\ref{thm:main} can be constructed from the Ginzburg-Landau $\nabla \phi$-interface model (see \cite{Fu05}), a well established  model for an interface separating two pure thermodynamical phases. The interface is  described by a random field of height variables  $\phi=\{\phi(x); x\in \bbZ^d \}$ sampled from a Gibbs measure formally given by  $Z^{-1} \exp(-H(\varphi)) \, \prod_{x\in \bbZ^d} d \varphi(x)$ with formal Hamiltonian $H(\varphi)= \sum_{e\in E_d} V(\nabla \varphi(e))$ and potential  $V\in  C^2(\bbR; \bbR_+)$, which we suppose to be even and strictly convex. Note that in the special case $V(x)=\frac 1 2 x^2$, the field  $\phi$ becomes a discrete Gaussian free field.  In $d\geq 3$ this can be made rigorous by taking the thermodynamical limit, while in dimension $d\geq 1$ one considers the gradient process instead. 
Then, thanks to the strict convexity we have the Brascamp-Lieb inequality, which allows one to show that any environment with random conductances of the form $\{\om(x,y)=\lambda(\nabla \phi(e)), e\in E_d\}$ for any positive, even, globally Lipschitz  function $\lambda\in C^1(\bbR)$ satisfies the spectral gap condition in  Assumption~\ref{asm:mixing}-(A2), see \cite[Section~7]{AN19} for details.
The Brascamp-Lieb inequality also implies  that exponential moments for gradient fields under the Gibbs measure exist (cf.\ \cite{Fu05, NS97}). Thus, the environment  $\{\om(e), e\in E_d\}$ as chosen above also satisfies the required moment condition in Theorem~\ref{thm:main}. The assumption of a strictly convex potential can be relaxed, see \cite{AT19}.

The $\nabla \phi$ interface model also satisfies the FKG inequality, see again e.g.\ \cite{Fu05, NS97}, and for models with \emph{massive} Hamiltonians formally given by
\begin{align*}
H(\varphi)= \sum_{e\in E_d} V(\nabla \varphi(e))+ \frac {m^2}2 \sum_{x\in \bbZ^d} \phi(x)^2, \qquad m>0,
\end{align*} 
 we have exponential correlation decay, see \cite[Theorem~B]{NS97}. In particular, Assumption~\ref{asm:mixing}-(A1) holds, and Theorem~\ref{thm:main} applies, for instance, to conductances of the form $\om(x,y)=\exp(\phi(x)+\phi(y))$, $\{x,y\} \in E_d$.
\end{example}

\subsection{The Method} \label{sec:intro_method}
%We will now the present the strategy of the proof in the context of the existing literature on heat-kernel bounds, and in particular how the proof depends on Assumption \ref{asm:mixing}.
It is well known that Gaussian lower and upper bounds on the heat kernel are equivalent in many situations to a parabolic Harnack inequality (PHI), e.g. in the case of uniformly elliptic conductances, see \cite{De99}. Indeed, the PHI implies near-diagonal bounds which are then converted into off-diagonal bounds via the established \textit{chaining method} (see e.g.\ \cite{FS86, De99, Ba04}). 

In our context,  a PHI has been obtained in \cite{harnack}.  Unfortunately, due to the special structure of the constant in the PHI in the case of unbounded conductances (see \eqref{eq:CPH_expl} below), in particular its dependence on $\|\mu^{\om}\|_{p,B(x,n)}$ and $\|\nu^{\om}\|_{q,B(x,n)}$, we cannot directly deduce off-diagonal Gaussian lower bounds from it. 
In order to get effective Gaussian off-diagonal bounds using the chaining argument, one needs to apply the Harnack inequality on a number of balls with radius $n$ over a distance of order $n^2$. In general, however, the ergodic theorem does not 
give the required uniform control on the convergence of space-averages of stationary random variables over such balls (see \cite{AJ75}). Therefore, in order to obtain lower Gaussian bounds  we will need to make use of one of the additional conditions on the correlations stated in Assumption \ref{asm:mixing}. 
Specifically, in Proposition~\ref{prop:dbounds} we employ any one of these conditions to derive a certain concentration estimate. Then, in Proposition~\ref{prop:sumbound} (and Corollary~\ref{cor:kappasbound}), we manipulate these to give us the desired uniform control on the space-averages of the conductances over the aforementioned chain of balls of radius $n.$ Finally, we utilize this uniform control within the chaining argument to yield the desired Gaussian off-diagonal lower bound.

Near-diagonal heat kernel bounds can also be deduced from a local limit theorem, cf.\ \cite[Lemma~5.3]{harnack}. Recently, such local limit theorems have been derived for a more general class of RCMs in \cite{ACS20,AT19} via De~Giorgi's iteration technique, circumventing the need for a PHI. However, the bounds obtained from arguments  in \cite{ACS20,AT19} involve random constants which are implicit functions of the  averages $\|\mu^{\om}\|_{p,B(x,n)}$ and $\|\nu^{\om}\|_{q,B(x,n)}$, while the chaining argument requires the more explicit dependence on the averages in the PHI in \cite{harnack}. Note that in \cite{harnack} the PHI has only been derived for the CSRW, so we obtain the lower heat kernel bounds in Theorem~\ref{thm:main} for the CSRW only, while the upper bounds in \cite{ADS19} have been established for a general class of speed measures.

\smallskip

The rest of the paper is organised as follows. In Section~\ref{sec:conc_est}, we first deduce some concentration estimates from the correlation decay conditions in Assumption~\ref{asm:mixing}, which are then used in Section~\ref{sec:pf_mainthm} to prove the lower Gaussian bounds in Theorem~\ref{thm:main}. Finally,  in Section~\ref{sec:green} we show Theorem~\ref{thm:green}.

 \section{Concentration estimates under decorrelation assumptions} \label{sec:conc_est}

Recall that $\bar{\mu}_p\ldef \mean[ \mu^\om(0)^p]$ and $\bar{\nu}_q \ldef \mean[ \nu^\om(0)^q]$ for any $p,q \in[1,\infty)$. In this section we will derive some moment estimates on the deviations of $\mu^\om(x)$ and $\nu^\om(x)$ from their means under Assumption~\ref{asm:mixing}. For that purpose, we define the centred random variables 
\[
 \Delta \mu_p^\om(x)\ldef \mu^\omega(x)^p-\bar{\mu}_p, \qquad \Delta \nu_q^\om(x)\ldef \nu^\omega(x)^q-\bar{\nu}_q , \qquad x\in \bbZ^d,
\]
for any $p,q \in [1,\infty)$ such that   $\bar{\mu}_p$ and  $\bar{\nu}_q $ are finite. 
 Our moment bounds on $\Delta \mu_p^\om$ and $\Delta \nu_q^\om$ will take the form given in the following definition.
\begin{definition}
For any $p,q \in [1,\infty)$ and $1\leq \theta <\eta <\infty$ we say that $\pr$ satisfies a $(p,q,\eta,\theta)$-moment bound, if there exists $c\in(0,\infty)$ such that 
\begin{align}  \label{eq:dbounds}
\mean\bigg[ \Big| \sum_{x\in R} \Delta \mu_p^\omega(x)\Big|^{\eta} \bigg] \leq c \abs{R}^\theta\qquad \text{and}\qquad 
	\mean\bigg[ \Big| \sum_{x\in R} \Delta \nu_q^\omega(x)\Big|^{\eta} \bigg] \leq c \abs{R}^\theta
\end{align}
for all hyper-rectangles $R\subset\Z^d$.
\end{definition}

 In the next proposition, which is the main result in this section, we gather and derive the relations between Assumption~\ref{asm:mixing} and $(p,q,\eta,\theta)$-moment bounds.

\begin{proposition} \label{prop:dbounds}
	Let $\zeta,p,q \in [1,\infty)$ and let   $R\subset \Z^d$ be a hyper-rectangle. Suppose that Assumptions~\ref{asm:ergodic} and \ref{asm:mixing} hold, and that $\zeta<\gamma/d$ if under (A1).  There exist constants $p_0,q_0,\eta,\theta \in [1,\infty)$ with $\eta-\theta\geq\zeta$ such that if $\omega(e)\in L^{p_0}(\pr)$, $\omega(e)^{-1}\in L^{q_0}(\pr)$ for any $e\in E_d$,  then the $(p,q,\eta,\theta)$-moment bound holds.
\end{proposition}

We will prove Proposition \ref{prop:dbounds} under each of the assumptions separately, referencing the necessary materials before incorporating them into the proof.
The following lemma is easily implied by \cite[Corollary~1]{rohtua}.
\begin{lemma} \label{lem:Bulinskii}
Let $\{Y(x):x\in\Z^d\}$ be a random field satisfying the FKG inequality, and which is stationary with respect to translation, and suppose that
\[
\sum_{|x| \geq n} \cov\big(Y(0), Y(x)\big) = O(n^{-\nu}) \qquad \text{and} \qquad \mathbb{E}\Big[\abs{Y(0)}^{\eta+\delta}\Big]<\infty,
\] for some $\delta,\nu>0$ and $\eta>2$. Then, for any hyper-rectangle $R\subset\Z^d,$
\[
\mean\Bigg[ \Big| \sum_{x \in R} Y(x)\Big|^{\eta}\bigg] = O(n^\theta),
\]
for $\theta>\max\{\eta/2,\chi(1-d^{-1}\min\{1,\nu \delta/\chi\}) /(\eta+\delta-2)\}$, where $\chi=\delta+(\eta+\delta)(\eta-2).$
\end{lemma}
\begin{proof}
	We apply stationarity and the positivity of covariances due to the FKG inequality to \cite[Corollary~1]{rohtua} to give the result. Indeed, by stationarity  any hyper-rectangle can be shifted into $\N^d$, and positivity of the covariances allows us to bound the summation of covariances over $\N^d$ by the summation of covariances over $\Z^d$.
\end{proof}

\begin{proof}[Proof of Proposition \ref{prop:dbounds} under (A1):]
	We will deal with the moment bound on the summation of the $\Delta\mu_p^\omega(x).$ The argument for the $\Delta\nu_q^\omega(x)$ follows identically.
	We will apply Lemma \ref{lem:Bulinskii} to the field $Y(x)=\mu^\omega(x)^p$, $x\in\Z^d$. Then,
	\[\sum_{|x| \geq n} \cov\big(Y(0), Y(x)\big) \leq c   \sum_{|x| \geq n} \abs{x}^{-\gamma}  \leq c n^{-(\gamma-d)}, \]
	where in the first inequality we used the polynomial mixing condition in (A1).  We can therefore take $\nu=\gamma-d$ in Lemma~\ref{lem:Bulinskii}.
	
	We now let $\eta=d\zeta$ and $p_0=p\alpha$ with $\alpha>\frac{d\zeta(d\zeta-2)}{\gamma-d\zeta}+d\zeta$. Then in Lemma~\ref{lem:Bulinskii} we take $\delta=\alpha-d\zeta$ and note that $\nu>d\zeta-d$, to give that \eqref{eq:dbounds} holds with any $\theta>\chi(1-d^{-1}\min\{\nu \delta /\chi,1\})/(\eta+\delta-2)$. A computation then yields $\eta-\theta>\zeta$ for $\theta$ chosen close enough to the lower bound above.
\end{proof}

We now turn to Proposition \ref{prop:dbounds} under Assumption~(A2). First, we recall that under the spectral gap condition, we have the following $p$-version of the spectral gap estimate. For $p\geq 1$ and any  $f \in L^{2p}(\Omega, \prob)$ with $\mean[f]=0$,
	\begin{align} \label{eq:sg_p}
	\mean\big[|f|^{2p}\big] \leq c(p,C_{\mathrm{sg}})
	&\mean \bigg[ \Big( \sum_{e\in E_d} \big( \partial_e f \big)^2
	\Big)^{\! p} \bigg],
	\end{align}
	which basically follows by applying \eqref{SG} to the function $|u|^{p}$, see \cite[Lemma~2]{GNO15}.

\begin{proof}[Proof of Proposition \ref{prop:dbounds} under (A2):]
We will follow a similar argument given in \cite[Lemma~2.10]{AN19}. Again, we will only show the moment estimate for $\Delta \mu^\om_p$.
Take $p_0=2\zeta p$. Noting that $f\ldef \sum_{y\in R}\Delta\mu_p^\om(y)$ has mean zero, we use the spectral gap estimate in the form \eqref{eq:sg_p} which  yields
\[
\mean\bigg[ \Big| \sum_{y\in R} \Delta\mu_p^\om(y)\Big|^{2\zeta} \bigg] \leq c \, \mean\!\bigg[ \Big(\sum_{e\in E_d} \big| \partial_e u \big|^2\Big)^{\! \zeta} \bigg].
\]
Now we observe that, for any $e=\{\bar{e}, \ubar{e} \}\in E_d$,
\[
\partial_e \big[ \Delta\mu_p^\om (y) \big]=\partial_e\big[\mu^\om(y)^p\big] = p \, \mu^\omega(y)^{p-1} \, \indicator_{\{\bar{e},\ubar{e}\}}(y),
\]
so that
\begin{align*}
\partial_e u \leq \begin{cases}
p \big( \mu^\om(\ubar{e})^{p-1} + \mu^\om(\bar{e})^{p-1} \big) & \text{if $\ubar{e} \in R$ or $\bar{e} \in R$,} \\
0 & \text{else}.
\end{cases}	
\end{align*}
Hence, 	
\begin{align*}
& \mean\bigg[ \Big| \sum_{y\in R} \Delta\mu_p^\om(y)\Big|^{2\zeta} \bigg]  \leq c 
\abs{R}^{ \zeta} \mean\big[ \mu^\om(0)^{2\zeta (p-1)} \big],
\end{align*}
and so we have obtained the requisite moment bounds with $\eta/2=\theta=\zeta.$
\end{proof}

\begin{lemma} \label{lem:Rosenthal}
Let $p\in (2,\infty)$. There exists a constant $c_{12}=c_{12}(p)$ such that if $Y_1,\ldots Y_n\in L^p(\prob)$ are independent random variables satisfying $\E{Y_j}=0$ for all $j\in\{1,\ldots,n\}$, then 
\[
\mean \bigg[ \Big|\sum_{i=1}^{n} Y_{j}\
\Big|^{p} \bigg]^{1 / p} \leq c_{12} \, \max 
\Bigg\{ \bigg(\sum_{j=1}^{n} \mathbb{E} \left[ \big|Y_{j}\big|^{p} \right] \bigg)^{\! 1 / p},\bigg(\sum_{j=1}^{n} \mathbb{E}\left[ \big|Y_{j}\big|^{2} \right] \bigg)^{\! 1 / 2}\Bigg\}
.\]
\end{lemma}
\begin{proof}
This can be extracted from \cite[Theorem~3]{rosenthal1970subspaces}.
\end{proof}

\begin{proof}[Proof of Proposition \ref{prop:dbounds} under (A3):]
Take $p_0=2p\zeta$, again considering only the moment bound on the sum of the $\Delta \mu_p^\omega(x)$ as the argument for $\Delta \nu_q^\omega(x)$ is the same.
Let $(e_i)_{1\leq i\leq d}$ denote the standard unit vectors. We call two vertices $x,y\in R$ equivalent if $x-y=\pm \mathfrak{R}e$ for some $e\in\{e_1,\ldots,e_d\}.$ Write the equivalence classes as $E_1,\dots,E_m,$ and observe that we must have $m\leq\mathfrak{R}^d$. Note that the size of each equivalence class is trivially bounded above by $\abs{R}.$ We apply the finite range assumption to give that for each fixed $i$, the $(\mu^\om(x))_{x\in E_i}$ are mutually independent, and therefore
\begin{align*}
& \mean\bigg[ \Big| \sum_{y\in R(x)} \Delta\mu_p^\om(y)\Big|^{2\zeta} \bigg] = \mean\bigg[ \Big| \sum_{i\leq m}\sum_{y\in E_i} \Delta\mu_p^\om(y)\Big|^{2\zeta} \bigg]  
\\
& \mspace{36mu}
\leq c\sum_{i\leq m}\mean\bigg[ \Big| \sum_{y\in E_i} \Delta\mu_p^\om(y)\Big|^{2\zeta} \bigg]\leq c \, \abs{R}^\zeta,
\end{align*}
where in the final step we apply Lemma~\ref{lem:Rosenthal} for each $i$ in the summation, with $(Y_j)$ an enumeration of $(\Delta\mu_p^\om(y))_{y\in E_i}.$ Thus \eqref{eq:dbounds} holds with $\eta/2=\theta=\zeta.$
\end{proof}

\begin{lemma}\label{lem:shao}
Let $\left\{Y_{i}, 1 \leqslant i \leqslant n\right\}$ be a negatively associated sequence. Further,  let $\left\{Y_{i}^{*}, 1 \leqslant i \leqslant n\right\}$ be a sequence of independent random variables such that $Y_{i}^{*}$ and $Y_{i}$ have the same distribution for each $i=1,2, \ldots, n .$ Then
$$
\mean\bigg[ \phi \bigg(\sum_{i=1}^{n} Y_{i}\bigg) \bigg] \leq \mean\bigg[ \phi \bigg(\sum_{i=1}^{n} Y_{i}^*\bigg) \bigg]
$$
for any convex function $\phi$ on $\R$, whenever the expectation on the right hand side exists.
\end{lemma}
\begin{proof}
This follows from \cite[Theorem~1]{shao2000comparison}.
\end{proof}

\begin{proof}[Proof of Proposition~\ref{prop:dbounds} under (A4):]
Let $p_0=2p\zeta$ and $q_0=2q\zeta$. Then, we apply Lemma~\ref{lem:shao} and  Lemma~\ref{lem:Rosenthal}, with $(Y_i)$ an enumeration of $(\Delta\mu_p^\om(y)^*)_{y\in R}$ (and $(\Delta\nu_q^\om(y)^*)_{y\in R}$, respectively) to give \eqref{eq:dbounds} with $\eta/2=\theta=\zeta$.
\end{proof}

As a first consequence of the concentration estimate in Proposition~\ref{prop:dbounds} we record the following tail estimate on the random variables $N_1(x)$, $x\in \bbZ^d$, defined via \eqref{eq:defN}.

\begin{lemma} \label{lem:NTailBounds}
	Suppose that Assumption~\ref{asm:ergodic} holds and that $\pr$ satisfies a $(p,q,\eta,\theta)$-moment bound, with $\zeta:=\eta-\theta>0$. Then there exists $c_{13}\in(0,\infty)$ such that
	\begin{equation} \label{eq:NTailBound}
	\pr\big(N_1(x)>n\big)\leq c_{13} \, n^{1-d\zeta}, \qquad \forall n\in\bbN.
	\end{equation}
\end{lemma}
\begin{proof}
	Note that, for any $n\in\bbN$, we have by a union bound
		\begin{align} \label{eq:NUnion}
\pr(N_1(x)>n)\leq \sum_{m\geq n} \Big(
\pr\big[\norm{\mu^\omega}_{p,B(x,m)}^p>2\bar{\mu}_p \big] + \pr\big[ \norm{\nu^\omega}_{q,B(x,m)}^q>2\bar{\nu}_q \big] \Big).
\end{align}
For the first term we get by Proposition~\ref{prop:dbounds} and Markov's inequality,
\begin{align*}
 & \sum_{m\geq n} \pr\big[\norm{\mu^\omega}_{p,B(x,m)}^p>2\bar{\mu}_p \big]
 =
  \sum_{m\geq n} \pr\bigg[\Big |\sum_{y\in B(x,m)} \Delta \mu_p^\omega(y)\Big|^{\eta}> \bar{\mu}_p^\eta \, \big|B(x,m)\big |^\eta\bigg] \\
  & \mspace{36mu}
\leq c \sum_{m\geq n} m^{-d\zeta}\leq c \, n^{1-d\zeta}.
	\end{align*}
Repeating the same argument with the second term in \eqref{eq:NUnion} gives the claim.
\end{proof}

\section{Heat kernel lower bounds} \label{sec:pf_mainthm}

We first recall the near-diagonal heat kernel bound in \cite[Proposition 4.7]{harnack}, which will be a key ingredient in the proof of the main theorem.

\begin{proposition} 	\label{prop:harnack}	
	Suppose that Assumption \ref{asm:ergodic} holds, and suppose there exist $p,q\in (1,\infty]$ with $1/p+1/q<2/d$ such that  $\omega(e)\in L^p(\pr)$ and $\omega(e)^{-1}\in L^q(\pr)$ for any $e\in E_d$. Then there exists $c_{14}=c_{14}(d)$ such that for any $t\geq 1$, $x_{1} \in \bbZ^d$ and $x_{2} \in B(x_{1}, \frac{1}{2} \sqrt{t}),$
	\begin{equation}\label{eq:lowerondiag}
	p^{\omega}\left(t, x_{1}, x_{2}\right) \geq \frac{c_{14}}{C_{\mathrm{PH}}} t^{-\frac{d}{2}},
	\end{equation}
	where $C_{\mathrm{PH}}=C_{\mathrm{PH}} \big(\|\mu^{\omega}\|_{p, B(x_{1}, \sqrt{t})},\|\nu^{\omega}\|_{q, B(x_{1}, \sqrt{t})}\big)$ is the constant appearing in the	parabolic Harnack inequality in \cite[Theorem~1.4]{harnack}, more explicitly given by
	\begin{align} \label{eq:CPH_expl}
	C_{\mathrm{PH}} \big( \|\mu^{\omega}\|_{p, B},\|\nu^{\omega}\|_{q, B}\big)=c \exp \Big(c \, \big(1 \vee \|\mu^{\omega}\|_{p, B}\big)^{\kappa}\big(1 \vee \|\nu^{\omega}\|_{q, B}\big)^{\kappa}\Big)
	\end{align}
	for some positive $c=c(d, p, q)$ and $\kappa=\kappa(d, p, q)\geq1$. 
\end{proposition} 
 
Theorem~\ref{thm:main} will be proven by the well-established chaining technique. More precisely, we will apply Proposition~\ref{prop:harnack} on a certain sequence of balls. 
 Given a vertex $x=(x_1,\ldots, x_d)\in \bbZ^d$ and  $0<r\leq 4|x|$, we specify a nearest-neighbour path $P[x]$ of length $D:=|x|$ from $0$ to $x$. Setting $p_0(x)\ldef 0$ and $p_i(x)\ldef (x_1, \ldots, x_i, 0, \ldots, 0) \in \bbZ^d$, $1\leq i \leq d$, we define $P[x]$ to be the path that consists of $d$ consecutive straight line segments connecting $p_0(x), p_1(x), \ldots, p_d(x)$. Next, for any $k\in \bbN$ with $\frac{12D}{r}\leq k\leq \frac{16D}{r}$,  we choose a subset  $\{z_0, \ldots z_k\} \subset P[x]$ such that $z_0=0$, $z_k=x$, $ d(z_j,z_{j-1})\leq \frac{r}{12}$ for $1\leq j \leq k$ and such that, for each $j\leq k$, $\big| B(z_j,r) \cap \{z_0, \ldots z_k\} \big| \leq c $ 
 for some $c=c(d)$.  Set $B_j\ldef B(z_j,r/48)$. Finally, we let $s\ldef Dr/k,$ then $\frac{1}{16}r^2\leq s \leq \frac{1}{12} r^2.$

\begin{proposition} \label{prop:sumbound}
	
Suppose that Assumption~\ref{asm:ergodic} holds. Further, for fixed $p,q \in [1, \infty)$ assume that $\pr$ satisfies a $(p,q,\eta,\theta)$-moment bound with $\zeta:=\eta-\theta>d$.
	Then there exist constants $c_{15},c_{16}\in (0,\infty)$ and a random variable $N_3=N_3(\om)$ satisfying 
	\begin{align} \label{eq:tailR0}
	\pr(N_3>\rho)\leq c_{15} \rho^{-d(\zeta-1)+2}, \qquad \forall \rho>0,
	\end{align}
	 such that, $\prob$-a.s., for all $r\geq N_3$, $x\in \bbZ^d$ and $(B_j)_{1\leq j \leq k}$ defined as right above, the following holds.
If $r \leq 4|x|$, for any collection of vertices $y_0, \ldots, y_k$ with $y_0=0$,  $y_j\in B_j$ for $1\leq j \leq k-1$ and $y_k=x$ we have 
	\begin{equation} \label{eq:sumbound}
	\sum_{j=0}^{k-1}\left(1 \vee\left\|\mu^{\omega}\right\|_{p, B(y_j,\sqrt{s})}\right)\left(1 \vee\left\|\nu^{\omega}\right\|_{q, B(y_j,\sqrt{s})}\right)\leq c_{16} \, k.
	\end{equation}
\end{proposition}

\begin{proof}
Set $B_{y_j}\ldef B(y_j,\sqrt{s})$ for abbreviation.
Then note that there exists $c=c(d)\in(0,\infty)$ such that
\begin{align} \label{eq:compBalls}
\big| \big\{ j\in \{1,\ldots k \}: z\in B_{y_j} \big\} \big| \leq c, \qquad \forall z\in \bigcup_{j\leq k} B_{y_j}.
\end{align}
We divide the rest of the proof into several steps.

\medskip
\emph{Step~1.}
For  $x\in \bbZ^d$ and $r$ as in the statement,  we will define a collection $(\cR^{x,r}_i)_{0\leq i \leq d}$ of $d+1$ hyper-rectangles in $\bbZ^d$ which covers  $\bigcup_{j\leq k} B_{y_j}$ for any selection $y_j\in B_j$. For simplicity, we will only give the definition for $x\in \bbZ^d \cap [0,\infty)^d$ -- it can be easily adjusted to the other regions of $\bbZ^d$.  
 For any $m, l\in \N$, $u\in \bbZ^d$ and $i=1,\ldots d$ we  write
\begin{align*}
R_i(u,m,l) \ldef u + \big\{v\in\Z^d:0\leq v_i\leq l,  \, \abs{v_j}\leq m, \text{ for all } j\neq i \big\}
\end{align*} 
  for the $d$-dimensional hyper-rectangle with base point $u$ and dimension $l$ along the $e_i$ axis and $m$ along the remaining coordinate axes.  
Now define  
\begin{align*}
\mathcal{R}_0^{x,r}\ldef \big([0,r] \times[-r,0]^{d-1}\big)\cap \bbZ^d \quad \text{and} \quad  \mathcal{R}_i^{x,r}\ldef R_i\big(p_{i-1}(x), r,x_i+r\big), \quad 1\leq i\leq d.
\end{align*}
  Then note that $\bigcup_{0\leq i\leq d} \mathcal{R}_i^{x,r}\supseteq \bigcup_{j\leq k} B_{y_j}$. 

\medskip

\emph{Step 2.} In this step we will show that there exists a random $N_3$ satisfying \eqref{eq:tailR0} such that for all $x\in \bbZ^d$ and all $r\geq N_3$,
\begin{align} \label{eq:sum_muNorm}
\sum_{j\leq k}\norm{\mu^\omega}_{p,B_{y_j}}^p\leq c \, k,\qquad\text{and}\qquad \sum_{j\leq k}\norm{\nu^\omega}_{q,B_{y_j}}^q\leq c \, k.
\end{align}
We will only discuss the first inequality as the arguments for the second are identical. 
By \eqref{eq:compBalls} and the fact that  the hyper-rectangles $(\cR^{x,r}_i)_{0\leq i \leq d}$ cover $\bigcup_{j\leq k} B_{y_j}$, we have that  
\begin{align*} \label{eq:split}
\sum_{j\leq k}\norm{\mu^\omega}_{p,B_{y_j}}^p
&
\leq c \,  r^{-d}\sum_{0\leq i\leq d}{\sum_{y\in \RR{i}{x}{r}}}\mu^\omega(y)^p\\
&\leq c \, 
 r^{-d}\bigg(\big(|x|+r \big)r^{d-1} \bar{\mu}_p+\sum_{0\leq i\leq d} \sum_{y\in \RR{i}{x}{r}}\Delta \mu_p^\om(y) \bigg)\\
&\leq c \, k + c \, r^{-d}\sum_{0\leq i\leq d}\sum_{y\in \RR{i}{x}{r}} \Delta \mu_p^\om(y),  \numberthis
\end{align*}
where we have used that $\abs{x} \leq c k r$.
Now we apply the moment-bound hypothesis with Proposition \ref{prop:dbounds} and Markov's inequality to give
\[
\pr\bigg(\sum_{0\leq i\leq d}\sum_{y\in \RR{i}{x}{r}} \Delta \mu_p^\om(y)> k r^d\bigg)\leq c \, \big( \abs{x}r^{d-1} \big)^{-\zeta},
\]
where we have used that $kr^d \geq c \abs{x} r^{d-1}$.
Now fix $\rho,l \in\N$ with $l\geq \rho$. By applying a union bound, and summing over $\partial B(l):=\{x\in \bbZ^d:\abs{x}=l\}$ and $r\geq \rho$, we get
\begin{align*}
\pr\bigg(\exists x\in \partial B(l), r\in [\rho,4l]\cap \bbN: \sum_{0\leq i\leq d}\sum_{y\in \RR{i}{x}{r}} \Delta \mu_p^\om(y)> k r^d\bigg) \leq c \, l^{d-1-\zeta} \rho^{-\zeta(d-1)+1}.
\end{align*}

Set 
\[
A_\rho:=\bigg\{\exists x\in \bbZ^d, r \in \bbN:  |x|\geq \rho, \,  r\in \big[\rho, 4 |x| \big],   \sum_{0\leq i\leq d}\sum_{y\in \RR{i}{x}{r}} \Delta \mu_p^\om(y)> k r^d\bigg\}.
\] Since  $\zeta>d$,  we can apply another union bound  over $l\geq \rho$ to obtain
\begin{equation} \label{eq:pBound}
\pr(A_\rho)\leq c \,  \rho^{d-\zeta} \rho^{-\zeta(d-1)+1}= c \, \rho^{d(1-\zeta)+1}.
\end{equation}

However, $d(1-\zeta)+1<-1$, so by Borel-Cantelli, for $\pr$-a.e.\ $\omega,$ there exists $N_3=N_3(\omega)$ such that $A_\rho$ does not occur  for $\rho\geq N_3$. Substituting into \eqref{eq:split} completes the proof of \eqref{eq:sum_muNorm}. Moreover, via a union bound, \eqref{eq:pBound} implies that $N_3$ can be constructed such that \eqref{eq:tailR0} holds.

\medskip 
\emph{Step~3.} By H{\"o}lder's inequality,
\begin{align*}
& \sum_{j\leq k}\left(1 \vee\left\|\mu^{\omega}\right\|_{p, B_{y_j}}\right)\left(1 \vee\left\|\nu^{\omega}\right\|_{q, B_{y_j}}\right)
\\
& \mspace{36mu}
\leq 
k^{1-\frac 1 p - \frac  1q} \, \bigg(\sum_{j\leq k} \Big( 1\vee {\norm{\mu^\omega}_{p,B_{y_j}}^p} \Big) \bigg)^{\!1/p}
\bigg(\sum_{j\leq k} \Big( 1\vee {\norm{\nu^\omega}_{q,B_{y_j}}^q} \Big)\bigg)^{\!1/q},
\end{align*}
so that the statement follows from Step~2.		
\end{proof}

\begin{remark}
	With a more convoluted covering argument, replacing union bounds with bounds on maxima, the requirement of $\zeta>d$ in Proposition \ref{prop:sumbound} can be decreased to $\zeta>d-1,$ and thus we only need $\gamma>d(d-1)$, and the minimal moment conditions of Remark \ref{rem:min} can be reduced to
		 $p_0>p\kappa\chi$ and $q_0>q\kappa\chi$ with $\chi=d(d-1)\big[1+\frac{d(d-1)-2}{\gamma-d(d-1)}\big]$ under (A1), and $p_0=2p\kappa (d-1)$ and $q_0=2q\kappa (d-1)$ under (A2), (A3) or (A4).
	We do not include this argument  as it brings greatly increased complication for very limited improvement.
\end{remark}

\begin{corollary} \label{cor:kappasbound}
In the setting of Proposition~\ref{prop:sumbound}, assume that $\pr$ satisfies a $(\kappa p,\kappa q,\eta,\theta)$-moment bound. Then there exists a  constant $c_{17}\in (0,\infty)$ and a random variable $N_4=N_4(\om)$ satisfying \eqref{eq:tailR0} such that, $\prob$-a.s., for all $r\geq N_4$, $x\in \bbZ^d$ with $r \leq 4|x|$, 
	\begin{equation} \label{eq:sumboundkappa}
\sum_{j=0}^{k-1}\left(1 \vee\left\|\mu^{\omega}\right\|_{p, B(y_j,\sqrt{s})}\right)^\kappa\left(1 \vee\left\|\nu^{\omega}\right\|_{q, B(y_j, \sqrt{s})}\right)^\kappa\leq c_{17} k.
\end{equation}
	\end{corollary}
	
	\begin{proof}
		This follows exactly as Proposition~\ref{prop:sumbound} after applying Jensen's inequality to $\left\|\mu^{\omega}\right\|_{p, B_{y_j}}$ and $\left\|\nu^{\omega}\right\|_{q, B_{y_j}}$, and then replacing $\mu^\omega$ by $(\mu^\omega)^\kappa$ and $\nu^\omega$ by $(\nu^\omega)^\kappa$.
\end{proof}

\begin{proof}[Proof of Theorem \ref{thm:main}]
 By translation invariance of the measure it suffices to show a lower bound on $p^\om(t,0,x)$. To begin with, we must establish the necessary moment conditions to deploy the tools developed in the previous section.
 
  We assume that there exist some $p,q\in (1,\infty)$ with $1/p+1/q<2/d$ such that $\omega(e)\in L^{p}(\prob)$ and $\omega(e)^{-1}\in L^{q}(\prob)$. This will allow us to apply Proposition~\ref{prop:harnack} involving the constant $\kappa=\kappa(p,q,d)$. If working under assumption (A1), recall that $\gamma>d^2$ and fix $d<\zeta<\gamma/d$; otherwise, just fix $\zeta>d$. Then Proposition~\ref{prop:dbounds} provides us with $p_0,q_0\geq 1$ such that if $\omega(e)\in L^{p_0}(\prob)$ and $\omega(e)^{-1}\in L^{q_0}(\prob)$, then $(1,1,\eta,\theta)$ and $(\kappa p,\kappa q,\eta,\theta)$-moment bounds hold with $\eta-\theta\geq \zeta$.
  This will allow us to apply  Lemma~\ref{lem:NTailBounds}, Proposition~\ref{prop:sumbound}, and Corollary~\ref{cor:kappasbound} as required. 
We then set $N\ldef N_1(0)\vee N_3 \vee N_4$ and combine the tail bounds in Lemma~\ref{lem:NTailBounds} and \eqref{eq:tailR0}  to obtain that $N$ satisfies the tail bound in \eqref{eq:mainTailBound}.

Set again $D\ldef \abs{x}$, and assume now that $t\geq N (D\vee 1)$. 
We will split the proof into two cases, $D^2/t\leq1/4$ and $D^2/t>1/4$.
 
 \medskip
 
 \emph{Case 1:}  $D^2/t\leq1/4$. Then $x\in B(0,\frac{1}{2}\sqrt{t})$, so by Proposition \ref{prop:harnack},
	\[
	p^{\omega}(t, 0, x) \geq \frac{c_{14}}{C_{\mathrm{PH}}} t^{-\frac{d}{2}}
	\]
	with $C_{\mathrm{PH}}= C_{\mathrm{PH}}\big(\|\mu^{\omega}\|_{p, B(0,\sqrt{t})},\|\nu^{\omega}\|_{q, B(0,\sqrt{t})}\big)$.
	Since  $C_{\mathrm{PH}}\big(\|\mu^{\omega}\|_{p, B},\|\nu^{\omega}\|_{q, B}\big)$ is increasing in $\norm{\mu^\omega}_{p,B}$ and $\norm{\nu^\omega}_{q,B}$ (cf.\ \eqref{eq:CPH_expl} above) and $t\geq N_1(0)$,
	\begin{align*}
	C_{\mathrm{PH}}\big(\|\mu^{\omega}\|_{p, B(0, \sqrt{t})},\|\nu^{\omega}\|_{q, B(0, \sqrt{t})}\big) \leq C_{\mathrm{PH}} \Big( \big(2\bar{\mu}_p\big)^{1/p}, \big(2\bar{\nu}_q\big)^{1/q} \Big),
	\end{align*}
	and therefore $p^{\omega}(t, 0, x) \geq c t^{-d / 2}$.  
	
\medskip	
	\emph{Case 2: } $D^2/t>1/4$. Set $r\ldef  t/D\geq 1\vee N_3 \vee N_4$. We deploy the chaining setup as introduced right  below Proposition~\ref{prop:harnack}. 
Recall that $s\ldef  D r/ k = t/k$ with  $\frac{12D}{r}\leq k\leq \frac{16D}{r}$ so that  $1\leq \frac{1}{16}r^2\leq s\leq \frac{1}{12}r^2$, and note that $k\geq 3$.
Then,  for any collection of vertices $y_0, \ldots y_k$ with $y_0=0$,  $y_j\in B_j$ for $1\leq j \leq k-1$ and $y_k=x$,  we have $d(y_i,y_{i+1})\leq r/8\leq \sqrt{s}/2$ so that by Proposition \ref{prop:harnack},
	\begin{align*}
	p^{\omega}(s, y_i, y_{i+1}) &\geq \frac{c_{14}}{C_{\mathrm{PH}}\big(\|\mu^{\omega}\|_{p, B_{y_j}},\|\nu^{\omega}\|_{q, B_{y_j}}\big)} \, s^{-\frac{d}{2}} 
	\geq
	 \frac{c}{C_{\mathrm{PH}}\big(\|\mu^{\omega}\|_{p, B_{y_j}},\|\nu^{\omega}\|_{q, B_{y_j}}\big)}  \,r^{-d},
	\end{align*} 
	with $B_{y_j}\ldef B(y_j,\sqrt{s})$. Further, recall the representation of $C_{\mathrm{PH}}$ in \eqref{eq:CPH_expl} and that $\mathrm{P}_{y_j}^\omega [X_s =y_{j+1}] =p^\om(s,y_j,y_{j+1}) \mu^\om(y_{j+1})$. Hence, by the Markov property,
	\begin{align*} 
&	\mathrm{P}_{0}^{\omega}\left[X_{t}=x\right] =\mathrm{P}_{0}^{\omega}\left[X_{k s}=x\right] \geq \mathrm{P}_{0}^{\omega}\left[X_{j s} \in B_{j}, 1 \leq j \leq k-1, X_{k s}=x\right] \\
	& \mspace{36mu}
	 \geq\sum_{y_1\in B_1,\dots,y_{k-1}\in B_{k-1}}
	\frac{ c^k  \left(\prod_{j=1}^{k-1} r^{-d} \mu^\om(y_j) \right)  s^{-d / 2} \mu^{\omega}(x)}{ \exp\Big(c\sum_{j=0}^{k-1}\big(1 \vee \|\mu^{\omega}\|_{p, B_{y_j}}\big)^{\kappa} \big(1 \vee \|\nu^{\omega}\|_{q, B_{y_j}}\big)^{\kappa}\Big)}
	\\
& \mspace{36mu}	\geq 
	c^k \bigg(\prod_{j=1}^{k-1} \| \mu^\om \|_{1,B_j}  \bigg) \, s^{-d/2}     \mu^\omega(x),
	\end{align*}
	where we used Corollary~\ref{cor:kappasbound} in the last step.
	 In particular,
	\[ \label{eq:pmain}
	p^{\omega}(t, 0, x) \geq c^k \, \bigg(\prod_{j=1}^{k-1} \| \mu^\om \|_{1,B_j}  \bigg) \, t^{-d/2}. \numberthis
	\]
Now, by the harmonic-geometric mean inequality and Jensen's inequality, we have
\begin{align*}
	 \bigg(\prod_{j=1}^{k-1} \| \mu^\om \|_{1,B_j}  \bigg)^{\frac{1}{k-1}} \geq \frac{k-1}{\sum_{j=1}^{k-1} \norm{\mu^\omega}_{1,B_j}^{-1}} \geq \frac{c \,(k-1)}{\sum_{j=1}^{k-1}\norm{\nu^\omega}_{1,B_j}}. 
	\end{align*}
We use Proposition~\ref{prop:sumbound} (setting $p=q=1$, $y_j=z_j$, and replacing $B_{y_j}$ with $B_j$) to obtain
\begin{align*}
\sum_{j=1}^{k-1}\norm{\nu^\omega}_{1,B_j} \leq c \, (k-1), 
\end{align*}
so that
\begin{align*}
\bigg(\prod_{j=1}^{k-1} \| \mu^\om \|_{1,B_j}  \bigg) \geq c^{k-1}.
\end{align*}
Combining this with \eqref{eq:pmain} yields $p^\om(t,0,x) \geq c \,  c_{18}^k t^{-d/2}$ for some $c_{18}\in(0,1)$, which gives the bound \eqref{eq:mainBound} by the choice of $k$. 
\end{proof}

\section{Green kernel estimates} \label{sec:green}
In this final section we utilize Theorems~\ref{thm:uhk} and \ref{thm:main} to establish Theorem~\ref{thm:green}. We refer to  \cite[Section~6]{BH09} for similar arguments.

\begin{proof}[Proof of Theorem~\ref{thm:green}]
(i)	First we deduce the upper bound \eqref{eq:GU} on the Green kernel. For any distinct $x,y\in\Z^d$, we decompose the integral as
	\begin{equation}\label{eq:decomp}
	g^\om(x,y)= 	\frac{1}{\mu^\omega(x)}\int_{0}^{N_1(x)^2} \textrm{P}_y^\omega(X_t=x)  \, \textrm{d}t + \int_{N_1(x)^2}^{N_{x,y}} p_t^\omega(x,y)\ \textrm{d}t + \int_{N_{x,y}}^{\infty} p_t^\omega(x,y)\ \textrm{d}t
	\end{equation}
	with $N_{x,y} \ldef N_1(x)^2\vee(\abs{x-y}/{c_1})$, where we used that $p_t^\om(x,y)=p_t^\om(y,x)$  by the symmetry of the heat kernel. Using Theorem~\ref{thm:uhk} we can bound the last two terms of \eqref{eq:decomp} by
\begin{align*}
 \int_{N_1(x)^2}^{N_{x,y}} p_t^\omega(x,y)\ \textrm{d}t  \leq c_2 e^{-c_4\abs{x-y}}\int_{1}^\infty t^{-d/2}\ \textrm{d}t \leq c \abs{x-y}^{2-d},
\end{align*}	
and 
\begin{align*}
\int_{N_{x,y}}^{\infty} p_t^\omega(x,y)\ \textrm{d}t \leq c_2 \int_0^\infty  t^{-d/2}\,  e^{-c_3 |x-y|^2/t}\ \textrm{d}t \leq  c \abs{x-y}^{2-d}.
\end{align*}	

	It is left to bound the first term in the right hand side of \eqref{eq:decomp}. Recall that the random walk $X$ spends i.i.d.\ $\mathrm{Exp}(1)$-distributed waiting times between its jumps. Set $\lambda \ldef N_1(x)^2$ and $r:=\abs{x-y}\geq 1$. In particular, the random walk starting at $y$ needs to perform at least $r$ jumps to get to $x$. Thus,
	\begin{align} \label{eq:po}
	&\frac{1}{\mu^\omega(x)}\int_{0}^{\lambda} \textrm{P}_y^\omega(X_t=x)  \, \textrm{d}t 
	\leq
	\frac{\lambda}{\mu^\omega(x)} \textrm{P}_y^\omega\big(X_t=x \text{ for any } t\in [0,\lambda] \big) \nonumber \\
	& \mspace{36mu} 	\leq \frac{\lambda}{\mu^\omega(x)}\mathrm{Pois}(\lambda)\big([r,\infty)\big).
	\end{align}
Here $\textrm{Pois}(\lambda)$ denotes the Poisson distribution with parameter $\lambda$, which we recall to have exponential tails (see e.g.\ \cite[Remark 2.6]{RandomGraphs}). So there exists ${N_2=N_2(\omega,x)}$  such that for each $y\in\Z^d$ with $\abs{x-y} \geq N_2(\omega,x)$ the first term in \eqref{eq:decomp} is bounded from above by  $c \abs{x-y}^{2-d}$, which completes the proof of \eqref{eq:GU}.
	
	\medskip 
	
(ii) 	 This follows directly from Theorem~\ref{thm:main}, which gives for $x,y\in\Z^d$ with $\abs{x-y}>N(x)$,
	\begin{align*}
	\int_{0}^{\infty} p_t^\omega(x,y)\ \textrm{d}t&\geq \int_{N(x) |x-y|}^{\infty} c_6 \, t^{-d/2} e^{-c_7 |x-y|^{2} / t}\ \textrm{d}t \geq \int_{ |x-y|^2}^{\infty} c_6 \, t^{-d/2}e^{c_7 |x-y|^{2} / t}\ \textrm{d}t\\
	& =|x-y|^{2-d}  \int_1^{\infty} c_6 \,  t^{-d/2} \, e^{-c_7/ t}\, \textrm{d}t = c \,  |x-y|^{2-d}.
	\end{align*}

(iii) 	First, we carry out some preparation for the proof of the upper bound. In particular, we show the Green kernel has finite second moments. By the symmetry of the heat kernel and the on-diagonal part of the upper bound in Theorem~\ref{thm:uhk},  note that
	\begin{align*}
	& g^\omega(x,y) = \int_0^{N_1(x)^2} p^\om_t(y,x)\ \mathrm{d}t+ \int_{N_1(x)^2}^\infty p^\om_t(x,y)\ \mathrm{d}t  \leq \frac{N_1(x)^2}{\mu^\omega(x)}+cN_1(x)^{2-d} \\
	& \mspace{36mu} \leq cN_1(x)^2\nu^\omega(x),
	\end{align*}
	where we used Jensen's inequality in the last step.
Assuming that $\omega(e)\in L^{p_0}(\prob)$ and $\omega(e)^{-1}\in L^{q_0}(\prob)$ for suitable $p_0,q_0\in(1,\infty)$, we apply Proposition~\ref{prop:dbounds}  together with Lemma~\ref{lem:NTailBounds} (with $\zeta=d$) to obtain  that $N_1\in L^{\beta}(\pr)$ for any $\beta< d^2-1$. Thus,  by H\"older's inequality,
\begin{align}
\label{eq:AGL2}
\mean\big[ g^\omega(x,y)^{\beta} \big] <\infty,
\end{align}
for any $\beta< (d^2-1)(q_0-1)/(2q_0)$. Then, as $d\geq 3$, assuming $q_0>2$ ensures that the second moment of the Green kernel exists.

\smallskip	
 
	We can now prove the upper bound of \eqref{eq:AGL}. To do so, we show that the random variable $N_2$ introduced in (i) satisfies the tail bound 
	\begin{align} \label{eq:tailN2}
	\prob\big[ N_2>u] \leq c\, u^{2-d}, \qquad \forall u\geq 1.
	\end{align}
	Indeed, if \eqref{eq:tailN2} holds true, then  we obtain
	\begin{align*}
	\mean\big[ g^\omega(x,y)\big] 
	\leq c\abs{x-y}^{2-d} + \mean\big[g^\omega(x,y)^2\big]^{1/2} \, \pr\big[N_2>|x-y|\big] 
	\leq c \, \abs{x-y}^{2-d},
	\end{align*}
	where we used  \eqref{eq:GU} and the Cauchy-Schwarz inequality in the first step,  and \eqref{eq:tailN2} and \eqref{eq:AGL2} in the second step. 
	
	In order to show \eqref{eq:tailN2}, recall that $N_2$ has been chosen as a value of $r$ such that
	$ \frac{\lambda}{\mu^\omega(x)}\mathrm{Pois}(\lambda)\big([r,\infty)\big) \leq c r^{2-d}$
	with $\lambda\ldef N_1(x)^2$. By Chernoff's inequality (cf.\ e.g.\ \cite[Corollary~2.4 and Remark~2.6]{RandomGraphs}), 	$\mathrm{Pois}(\lambda)\big([r,\infty)\big)\leq  e^{-r+7 \lambda}$ for $r>7\lambda$. Hence, $N_2$ can be chosen as a constant times $\lambda=N_1^2$, so the tail bound on $N_2$ can be dominated by a tail bound on $N_1^2$, which is  provided by 	 Proposition~\ref{prop:dbounds} and Lemma~\ref{lem:NTailBounds} (with the choice $\zeta=d$) under suitable  moment conditions on $\omega(e)$ and $\omega(e)^{-1}$. More precisely,
	\[
	\pr(N_2>u)\leq \pr(N_1^2>c u) \leq  c \,  u^{\frac{1-d^2}{2}} \leq c\, u^{2-d}, \qquad u\geq 1,
\]
since  $(1-d^2)/2< 2-d$  for $d\geq 3$, which completes the proof of \eqref{eq:tailN2}.

\smallskip	

Finally, we prove the lower bound of \eqref{eq:AGL}, which follows again from Theorem~\ref{thm:main}.  Choose $K\in(0,\infty)$ such that $\pr(N(x)\leq K)=\pr(N(0)\leq K)\geq 1/2$, then 
\[
\mean\big[g^\omega(x,y) \big]
\geq
\mean \Big[ \int_{N(x)}^\infty p_t^\omega(x,y)\  \textrm{d}t \Big] 
\geq \frac{1}{2}\int_{K}^{\infty}c_6 \, t^{-d/2} \, e^{-c_7 |x-y|^{2} / t}\  \textrm{d}t.
\]
If $\abs{x-y}^2 \geq K$ then we can bound the integral on the right hand side as in the proof of \eqref{eq:GL} to give \eqref{eq:AGL}. 
On the other hand, there are only finitely many vertices $z\in B(0,\sqrt{K})$, and for each such $z$ we have $\E{g^\omega(0,z)}>0$. Therefore, $\inf_{y\in B(x,\sqrt{K})}\E{g^\omega(x,y)}=\inf_{z\in B(0,\sqrt{K})}\E{g^\omega(0,z)}>0$. Thus, 
we can adjust the constant $c_{10}$ such that  \eqref{eq:AGL} also holds for $\abs{x-y}^2 \leq K$. 
\end{proof}

\subsubsection*{Acknowledgement} 
We thank Scott Armstrong for some valuable comments on an earlier version of this paper. N.H.\ has been supported by the doctoral training centre, Cambridge Mathematics of Information (CMI).

\bibliographystyle{abbrv}
\bibliography{bibliography}

\end{document}